\documentclass[12pt]{amsart}
\usepackage{amssymb, latexsym, amsthm} 

\newtheorem{thm}{Theorem}[section] 

\newtheorem{cor}[thm]{Corollary}
\newtheorem{defn}[thm]{Definition}

\newtheorem{lem}[thm]{Lemma}
\newtheorem{prop}[thm]{Proposition}

\newtheorem{rem}[thm]{Remark}

\newcommand\Cref[1]{{Corollary~\ref{#1}}}

\newcommand\ideal[1]{{\left<{#1}\right>}}
\newcommand\sg[1]{{\ideal{#1}}}
\def\isom{{\;\cong\;}}
\def\s{\sigma}

\def\Z {{\mathbb {Z}}}

\newcommand\mul[1]{{#1^{\times}}} 
\newcommand\operA[2]{{\if!#2!\operatorname{#1}\else{\operatorname{#1}_{#2}^{\phantom{I}}}\fi}} \newcommand\Norm[1][]{\operA{N}{#1}}
\newcommand{\set}[1]{{\left\{#1\right\}}}
\long\def\forget#1\forgotten{}

\renewcommand\H[4][!]{{\operatorname{H}^{#2}\!\!\;({#3},{#4}{\if!#1\relax\else(#1)\fi})}}

\DeclareMathOperator{\I}{Im}
\DeclareMathOperator{\K}{Ker}

\DeclareMathOperator{\M}{M}
\DeclareMathOperator{\C}{C}
\DeclareMathOperator{\Cent}{Cent}

\DeclareMathOperator{\Br}{Br}
\DeclareMathOperator{\Res}{res}

\DeclareMathOperator{\X}{X}

\DeclareMathOperator{\G}{G}
\DeclareMathOperator{\Gal}{Gal}

\def\N{{\operatorname{N}}}

\def\Z{\mathbb{Z}}
\def\s{\sigma}

\newcommand\res[1][{}] {{\operatorname{res}_{#1}}}
\newcommand{\Trace}[1][]{\if!#1!\operatorname{Tr}\else{\operatorname{Tr}_{#1}^{\phantom{I}}}\fi} 

\newcommand\book[4]{{{#1},\ {{#2}}{\if!#3!\relax\else{,\ {#3}}\fi}{\if!#4!\relax\else{,\ {#4}}\fi}.}} 

\newif\ifXY 
\XYtrue     
\ifXY

\input xy
\input xyidioms.tex
\usepackage{xy}
\xyoption{all} %
\fi 

\begin{document}
\title{Triple Massey products in Galois cohomology}
\author{Eliyahu Matzri}

\address{Department of Mathematics \\
         Ben-Gurion University of the Negev\\
         Be'er-Sheva 84105 \\
         Israel}
\email{elimatzri@gmail.com}
\thanks{The author was supported by the Israel Science Foundation (grant No.\ 152/13) and by the Kreitman foundation.}
\thanks{The author would like to thank I. Efrat for introducing him with this fascinating subject and for many interesting discussions and helpful suggestions.}

\begin{abstract}
Fix an arbitrary prime $p$. Let $F$ be a field, containing a primitive $p$-th root of unity, with absolute Galois group $G_F$.
The triple Massey product (in the mod-$p$ Galois cohomology) is a partially defined, multi-valued function $\langle \cdot,\cdot,\cdot \rangle: H^1(G_F)^3\rightarrow H^2(G_F).$
In this work we prove a conjecture made in \cite{MT} stating that any defined triple Massey product contains zero.
As a result the pro-$p$ groups appearing in \cite{MT} are excluded from being absolute Galois groups of fields $F$ as above.

\end{abstract}

\maketitle
%

\section{Introduction}

We fix a prime number $p$. Let $F$ be a field, which will always be assumed to contain a primitive $p$-th root of unity, $\rho$.
Let $G$ be a profinite group acting trivially on $\mathbb{Z}/p\mathbb{Z}$ and let $H^i(G)$ denote the $i$-th cohomology group $H^i(G,\mathbb{Z}/p\mathbb{Z})$. 
The triple Massey product is a partially defined, multi-valued function $\langle \cdot,\cdot,\cdot \rangle: H^1(G)^3\rightarrow H^2(G)$.
In particular, for $1\leq i\leq 3$, let  $\chi_i\in H^1(G)$, be such that $\chi_1\cup \chi_2=\chi_2\cup \chi_3=0$ in $H^2(G).$
Then one can define a subset, $\langle \chi_1,\chi_2,\chi_3 \rangle\subseteq H^2(G)$, called the triple Massey product of $\chi_1,\chi_2,\chi_3$.
In \cite{MT}, Min\'{a}\v{c} and T\^{a}n define the Vanishing triple Massey product property of $G$, stating that every defined
triple Massey product contains zero. They conjecture that if $G=G_F$, the absolute Galois group of $F$, then for any prime $p$, $G_F$ has this property.
When $p=2$ they prove the conjecture and use it to show certain pro-$2$ groups can not be realizable as absolute Galois groups.
The main objective of this work is to prove their conjecture in its full generality, i.e. for any prime $p$ and any field $F$ as above. This is achieved in Theorem \ref{Main}.
As a result we get a strong restriction on the structure of $G_F$, ruling out certain pro-$p$ groups (along the lines of \cite{MT})
from being realizable as absolute Galois groups.

The Vanishing triple Massey product for $G_F$ was known in the following cases:
\begin{enumerate}
\item In \cite{HW}, for $p=2$, Hopkins and Wickelgren construct an $F$-variety $\X (\chi_1,\chi_2,\chi_3)$ such that
$0\in \langle \chi_1,\chi_2,\chi_3 \rangle\Leftrightarrow \X_F (\chi_1,\chi_2,\chi_3)\neq \emptyset$
and prove that it always has an $F$-point when $F$ is a number field.
\item In \cite{MT}, for $p=2$, Min\'{a}\v{c} and T\^{a}n present an $F$-point of $\X (\chi_1,\chi_2,\chi_3)$ where $F$ is any field.
Thus proving $G_F$ always has the Vanishing triple Massey product property with respect to the prime $p=2$.
Using the work of Dwyer (see \cite{Dwyer75}) they give examples of pro-$p$ groups which do not possess the Vanishing triple Massey product property.
Thus, for $p=2$ they get new examples of profinite groups which can not be realized as the absolute Galois groups of the field $F$.
\item In \cite{MT2} Min\'{a}\v{c} and T\^{a}n prove that if $F$ is a number field, $G_F$ has the Vanishing triple Massey product property with respect to any prime $p$.
\item In \cite{EM}, Efrat and the author connect this property to the Brauer group of $F$. When $p=2$ they show that an old Theorem of Albert (see \cite{Alb}),
strengthened by Rowen (see \cite[Corollary 5]{Row}) implies the Vanishing triple Massey product property for $G_F$, where $F$ is arbitrary. Moreover, when $F$ is a number field they prove
a result about relative Brauer groups of abelian extensions of number fields which implies the Vanishing triple Massey product property for $G_F$ with
 respect to any prime $p$.
\end{enumerate}

Following \cite{EM} we approach the problem from the point of view of the Brauer group of $F$. Namely, we make use of the isomorphism:
\begin{equation}\label{BrH2}
  \nonumber \Psi:H^2(G_F)\overset{\sim}{\longrightarrow} \Br_p(F)
\end{equation}
to translate the problem to a question about specific $\mathbb{Z}/p\mathbb{Z}\times \mathbb{Z}/p\mathbb{Z}$ abelian crossed products and solve it using the theory
of abelian crossed products. In particular, let $\chi_1, \chi_2, \chi_3\in H^1(G_F)$, which
now correspond to classes $[a_1],[a_2],[a_3]\in F^{\times}/(F^{\times})^p$
under the Kummer isomorphism, $$F^{\times}/(F^{\times})^p\cong H^1(G_F).$$
Let $\M$ be any element of $\langle \chi_1,\chi_2,\chi_3 \rangle$. Then we have:
\begin{equation}\label{cond}
 0\in \langle \chi_1,\chi_2,\chi_3 \rangle \Leftrightarrow \Psi(\M)= [(a_1,\alpha)_p] \otimes [(a_3, \beta)_p]
\end{equation}

where $\alpha,\beta \in F^{\times}$ and $[(a_1,\alpha)_p], [(a_3, \beta)_p]$ are the classes in $\Br_p(F)$ of the symbol $F$-algebras
 $(a_1,\alpha)_p$ and $(a_3, \beta)_p.$
Our strategy will be to construct an explicit abelian crossed product $F$-algebra $A$, such that $\Psi(\M)=[A]$ for some element
$\M\in \langle \chi_1,\chi_2,\chi_3 \rangle,$ and prove it satisfies the right-hand side of (\ref{cond}).
The work is organized as follows:
In section $2$ we give some background on general Massey products and start the study of the specific case of triple Massey products.
In section $3$ we give the necessary background on abelian crossed products, construct and study a specific abelian crossed product $A$
which will be used in the proof of the main Theorem.
In section $4$ we connect the two previous sections and show that every element in the triple Massey product corresponds to an abelian crossed product in the Brauer group.
In Section $5$ we study the condition $\chi_a\cup\chi_b=0$ in Galois cohomology in order to better understand the corresponding abelian crossed product from section $4$.
In particular, we build a Galois extension with group $G$, and construct a concrete function $\varphi_{a,b}\in C^1(G_F)$ such that
$\partial(\varphi_{a,b})=\chi_a\cup\chi_b \in C^2(G_F)$.
Section $6$ is devoted to the proof of the Vanishing triple Massey product property of absolute Galois groups.
In particular, we use the function from section $5$ to show $A$ (constructed in section $3$) corresponds to some element in the triple Massey product and use our
study of $A$ in section $3$ to prove the main Theorem.

\section{$n$-fold Massey products}
In this section we give the necessary background on (rank $1$) $n$-fold Massey products $\langle \cdot,\dots,\cdot \rangle : H^1(G)^n\rightarrow H^2(G).$
\subsection{Background}
The main sources for this background subsection are \cite{Dwyer75}, \cite{IE}, \cite{HW} and \cite{Wic}.
Let $R$ be a unital commutative ring. Recall that a differential graded algebra (DGA) over $R$ is a graded $R$-algebra
$$ C^\bullet = \oplus_{k\geq 0}C^k = C^0 \oplus C^1\oplus C^2\oplus ... $$
with product $\cup$, and equipped with a differential $\partial: C^{\bullet} \rightarrow C^{\bullet+1}$~such~that:
\begin{enumerate}
\item $\partial(x\cup y)=\partial(x)\cup y+(-1)^k x\cup \partial(y)$ for $x\in C^k;$
\item $\partial^2=0.$
\end{enumerate}
One then defines the cohomology ring $H^{\bullet}=\K(\partial)/\I(\partial).$

\begin{defn}
Let $c_1,..., c_n\in H^1$. A collection $C = (c_{i,j}), 1 \leq i < j \leq n+1, (i, j) \neq (1, n + 1),$ of elements
of $C^1$ is called a defining system for the $n$-th fold Massey product $\langle c_1, . . . , c_n\rangle$ if the following
conditions hold:
\begin{enumerate}
\item $c_{i,i+1}$ represents $c_i$ for every $1\leq i \leq n$;
\item $\partial(c_{i,j})=\sum_{k=i+1}^{j-1}c_{i,k}\cup c_{k,j}$ for every $ i, j$ as above and $i+1 < j.$
\end{enumerate}
One can then check that $\sum_{k=2}^{n}c_{1,k}\cup c_{k,n+1}$ is in $\K(\partial)$, thus represents an element of $H^2$.
It is called the $n$-fold Massey product with respect to the defining system $C$, denoted $\langle c_1, . . . , c_n\rangle_C$.
Then $$\langle c_1, . . . , c_n\rangle=\{ \langle c_1, . . . , c_n\rangle_C| \ C \hbox{ \ is a defining system} \}.$$
\end{defn}
\begin{rem}
For $n=2$, $\langle c_1, c_2\rangle =\{c_1\cup c_2\}.$
\end{rem}

\subsection{Massey products and unipotent representations}

Let $G$ be a profinite group and let $R$ be a finite commutative ring considered as a trivial
discrete $G$-module. Let $C^{\bullet} = C^{\bullet}(G,R)$ be the DGA of inhomogeneous continuous cochains of $G$ with coefficients in $R$ \cite[Chapter I Section 2]{NSW}.
In \cite{Dwyer75}, Dwyer shows in the discrete context (see also \cite[section 8]{IE} in the profinite case), that defining systems for this DGA can be interpreted in terms
of upper-triangular unipotent representations of $G$, in the following way.

Let $\mathbb{U}_{n+1}(R)$ be the group of all upper-triangular unipotent $(n+1)\times(n+1)$-matrices over $R$. Let $Z_{n+1}(R)$ be the subgroup of all such
matrices with entries only at the $(1,n+1)$ position. We may identify $\mathbb{U}_{n+1}(R)/Z_{n+1}(R)$ with the group
$\bar{\mathbb{U}}_{n+1}(R)$ of all upper-triangular unipotent $(n+1)\times (n+1)$-matrices over $R$ where we remove the $(1, n + 1)$-entry.

For a representation $\varphi: G \rightarrow \mathbb{U}_{n+1}(R)$ and $1 \leq i < j \leq n+1$ let $\varphi_{i,j} : G \rightarrow R$ be the
composition of $\varphi$ with the projection from $\mathbb{U}_{n+1}(R)$ to its $(i, j)$-coordinate; and the same for a representation \\ $\bar{\varphi}: G\rightarrow  \bar{\mathbb{U}}_{n+1}(R).$

\begin{thm}\label{DW}(\cite[Theorem 2.4]{Dwyer75})
Let $c_1,. . . , c_n$ be elements of $H^1(G,R).$ There is a one-one correspondence $C \leftrightarrow \bar{\varphi}_C$ between defining systems $C$
 for $\langle c_1, . . . , c_n\rangle$ and group homomorphisms $\bar{\varphi}_C : G \rightarrow \bar{\mathbb{U}}_{n+1}(R)$ with $\bar{\varphi}_{i,i+1} = −c_i$,
 for $1 \leq i \leq n.$
Moreover, $\langle c_1, . . . , c_n\rangle_C = 0$ if and only if the dotted arrow exists in the following commutative diagram
\[
\xymatrix{
   &   &   &  G \ar@{.>}[ld]\ar[d]^{\bar{\varphi}_C} & \\
1\ar[r] & R \ar[r] & \mathbb{U}_{n+1}(R)\ar[r] & \bar{\mathbb{U}}_{n+1}(R)\ar[r] & 1.
}
\]

\end{thm}
\begin{defn}
Following \cite{MT}, for $n \geq 2$ we say that $G$ has the Vanishing $n$-fold Massey product property if every defined
Massey product $\langle c_1, . . . , c_n\rangle$, contains zero.
\end{defn}

\subsection{Triple Massey products in Galois cohomology}
Let $F$ be a field containing a primitive $p$-th root of unity, $\rho$, with  absolute Galois group $G_F$.
We now focus on the case $n=3$ where the DGA used is that of Galois cohomology, namely $C^{\bullet}=C^{\bullet}(G_F,\mathbb{Z}/p\mathbb{Z})$, abbreviated $C^{\bullet}(G_F).$
In this case the above construction can be described as follows: Let $ \chi_a,\chi_b,\chi_c\in H^1(G_F) $ such that
\begin{equation}\label{MC}
  \chi_a\cup\chi_b=\chi_b\cup\chi_c=0 \in H^2(G_F).
\end{equation}

Then there exist $\varphi=\{\varphi_{a,b},\varphi_{b,c}\} \subset C^1(G_F)$ such that \begin{enumerate}
                                                                              \item $\partial(\varphi_{a,b})=\chi_a\cup \chi_b$ in $C^2(G_F).$
                                                                              \item $\partial(\varphi_{b,c})=\chi_b\cup \chi_c$ in $C^2(G_F).$
                                                                            \end{enumerate}
We call such $\varphi$ a defining system.
Then $$\langle \chi_a,\chi_b,\chi_c \rangle_{\varphi}= \chi_a\cup\varphi_{b,c}+ \varphi_{a,b}\cup \chi_c.$$ Note that $$\partial(\langle \chi_a,\chi_b,\chi_c \rangle_{\varphi})=\partial( \chi_a\cup\varphi_{b,c}+ \varphi_{a,b}\cup \chi_c)=
\chi_a\cup \chi_b\cup \chi_c-\chi_a\cup \chi_b\cup \chi_c =0.$$ Hence $\langle \chi_a,\chi_b,\chi_c \rangle_{\varphi}$ represents a class in $H^2(G_F)$ which we also denote by $\langle \chi_a,\chi_b,\chi_c \rangle_{\varphi}$.

Then assuming condition (\ref{MC}) holds, the triple Massey product is:
$$\langle \chi_a,\chi_b,\chi_c \rangle= \{ \langle \chi_a,\chi_b,\chi_c \rangle_{\varphi}| \ \varphi \hbox{ \ is a defining system} \} \subseteq H^2(G_F).$$

\begin{rem}\label{Set}
Notice that once a defining system $\varphi=\{\varphi_{a,b},\varphi_{b,c}\}$ is chosen, $\varphi_{a,b},\varphi_{b,c}$ are unique up to elements of $H^1(G_F)$, thus the set of all possible $\langle \chi_a,\chi_b,\chi_c \rangle_{\varphi}$ is a
coset $$ \langle \chi_a,\chi_b,\chi_c \rangle=\langle \chi_a,\chi_b,\chi_c \rangle_{\varphi}+ \chi_a\cup H^1(G_F) + H^1(G_F)\cup \chi_c.$$
\end{rem}
\begin{proof}
Indeed by definition of cohomology, for $\chi\in H^1(G_F)$ one has $\partial(\chi)=0$, and vice versa.
\end{proof}
This clearly implies,
\begin{prop}\label{p0} Let $\varphi$ be any defining system for $\langle \chi_a,\chi_b,\chi_c \rangle$. Then,
$$0\in \langle \chi_a,\chi_b,\chi_c \rangle\Leftrightarrow \langle \chi_a,\chi_b,\chi_c \rangle_{\varphi}\in \chi_a\cup H^1(G_F) + H^1(G_F)\cup \chi_c.$$
\end{prop}

\section{Abelian crossed products}
In this section we give the necessary background on abelian crossed products, define and study a specific abelian crossed product which we will later use.
\subsection{Background}
A crossed product is a central simple algebra with a maximal subfield Galois over the center.
Such algebras, with an abelian Galois group, were studied by Amitsur and Saltman in \cite{AS}.
Our focus is in degree $p^2$, so let us summarize what we need from their work.
\begin{thm}
Let $K/F$ be a Galois extension of fields with Galois group $$\Gal(K/F) = \sg{\s_1,\s_2}\isom \Z/p\Z\times \Z/p\Z,$$ and let $b_1, b_2, u \in \mul{K}$
be elements satisfying the equations
\begin{eqnarray}\label{ACPC}
\s_i(b_i) & = & b_i, \nonumber \\
\s_2(b_1) b_1^{-1} & = & \Norm_{K/K^{\s_1}}(u), \label{Stand}\\
\s_1(b_2)b_2^{-1} & =&  \Norm_{K/K^{\s_2}}(u)^{-1}. \nonumber
\end{eqnarray}
Then the algebra $A = K[z_1,z_2]$, defined by the relations $z_i k z_i^{-1} = \s_i(k)$ for $k\in K$, $z_i^p = b_i$ and $z_2 z_1 = u z_1 z_2$,
is an $F$-central simple algebra containing $K$ as a maximal subfield, and every such algebra has this form.
\end{thm}
The crossed product defined above is denoted $(K/F,\set{\s_1,\s_2},\set{b_1,b_2,u})$.

\subsection{Construction of an abelian crossed product $A_{v_1,v_2}$}  \ \\
Let $a_1,a_2\in \mul{F}$ and let $F_i=F[x_i|x_i^p=a_i], i=1,2$, be the corresponding Kummer extensions of $F$, with Galois groups,
$G_i=\G(F_i/F)=\langle\s_i\rangle \cong \mathbb{Z}/p\mathbb{Z}.$
Let $K=F_1F_2$. Then $\G(K/F)=\langle\s_1,\s_2\rangle$. Note that $K^{\s_1}=F_2, \  K^{\s_2}=F_1$.
Assume that there exist elements $v_i\in F_i$, $i=1,2$, such that
\begin{equation}\label{eq1}
  \N_{F_{a_1}/F}(v_1)=\N_{F_{a_2}/F}(v_2).
\end{equation}
Define \begin{equation}\label{eq2}
 \nonumber u=\frac{v_2}{v_1}.
\end{equation}
Notice that $\N_{K/F}(u)=1.$
\begin{prop}\label{prop3}
Let  \begin{equation}\label{eq3}
\nonumber w_1=\prod_{i=0}^{p-1}\s_2^i(v_2)^i\in F_2 \hbox{, \ } w_2=\prod_{i=0}^{p-1}\s_1^i(v_1)^i\in F_1.
\end{equation}
Then we have: \begin{enumerate}
    \item $\N_{K/F_2}(u)=\frac{\s_2(w_1)}{w_1}.$
    \item $\N_{K/F_1}(u)^{-1}=\frac{\s_1(w_2)}{w_2}.$

\end{enumerate}

\end{prop}

\begin{proof}
    (1) \quad We compute
    \begin{eqnarray}
    \nonumber  \s_2(w_1) &=& \s_2(\prod_{i=0}^{p-1}\s_2^i(v_2)^i) \\
     \nonumber &=& \s_2^2(v_2)\s_2^3(v_2)^2 \cdot \cdot \cdot \s_2^{p-1}(v_2) ^{p-2} v_2 ^{p-1} \\
     \nonumber &=& \frac{\prod _{i=0}^{p-1} \s_2^i(v_2)}{\prod _{i=0}^{p-1} \s_2^i(v_2)}\s_2^2(v_2)\s_2^3(v_2)^2
     \cdot \cdot \cdot \s_2^{p-1}(v_2) ^{p-2} v_2 ^{p-1} \\
     \nonumber &=& \frac{v_2^p\prod _{i=0}^{p-1} \s_2^i(v_2)^i}{\N_{F_2/F}(v_2)}=\frac{v_2^p}{\N_{F_1/F}(v_1)}w_1=\N_{K/F_2}(\frac{v_2}{v_1})w_1=\N_{K/F_2}(u)w_1.
    \end{eqnarray}
  Thus, $\N_{K/F_2}(u)=\frac{\s_2(w_1)}{w_1}.$

(2) \quad We compute
    \begin{eqnarray}
    \nonumber  \s_1(w_2) &=& \s_1(\prod_{i=0}^{p-1}\s_1^i(v_1)^i)  \\
     \nonumber &=& \s_1^2(v_1)\s_1^3(v_1)^2 \cdot \cdot \cdot \s_1^{p-1}(v_1) ^{p-2} v_1 ^{p-1} \\
     \nonumber &=& \frac{\prod _{i=0}^{p-1} \s_1^i(v_1)}{\prod _{i=0}^{p-1} \s_1^i(v_1)}\s_1^2(v_1)\s_1^3(v_1)^2
     \cdot \cdot \cdot \s_1^{p-1}(v_1) ^{p-2} v_1 ^{p-1} \\
     \nonumber &=& \frac{v_1^p\prod _{i=0}^{p-1} \s_1^i(v_1)^i}{\N_{F_1/F}(v_1)}=\frac{v_1^p}{\N_{F_2/F}(v_2)}w_2=\N_{K/F_1}(\frac{v_1}{v_2})w_2= \N_{K/F_1}(u)^{-1}w_2.
    \end{eqnarray}
  Thus, $\N_{K/F_1}(u)^{-1}=\frac{\s_1(w_2)}{w_2}.$
  \end{proof}

We have proved:

\begin{prop}\label{example1}
Let $v_1,v_2,w_1,w_2\in K^{\times}$ be as in (\ref{eq1}), (\ref{eq3}) respectively. Consider the following data:
$$u=\frac{v_2}{v_1}, \ b_1=w_2, \ b_2=w_1.$$
Then all the conditions of (\ref{ACPC}) are satisfied and there exist an abelian crossed product, $$ A_{v_1,v_2}=(K/F,\set{\s_1,\s_2},\set{w_2,w_1,u}).$$
\end{prop}

We recall that this means that $A_{v_1,v_2}=K[z_1,z_2]$ such that the following relations hold,
\begin{align}\label{t}
  \nonumber & z_ikz_i^{-1} = \s_i(k) \hbox{ for } i=1,2 \hbox{ and all } k\in K, \\
  \nonumber & z_1^p =w_2, \ z_2^p=w_1,\\
  \nonumber & z_2z_1= uz_1z_2.
\end{align}

For the rest of this subsection we will show that $$A_{v_1,v_2}\cong (a,\alpha)_p\otimes (c,\beta)_p$$ for some $\alpha,\beta\in F^{\times}.$
 \begin{lem}\label{relations}
    \begin{enumerate}
      \item Let $g=\sigma_1 \sigma_2$, then $\N_{K/K^{\langle g\rangle}}(u)=1.$
      \item There exist $t\in K$ such that $u=\frac{g(t)}{t}.$
    \end{enumerate}
\end{lem}
\begin{proof}  \ \\
    (1) \quad We compute:
      \begin{eqnarray}
          \nonumber    \N_{K/k^{\langle g\rangle}}(u)&=&\frac{\N_{K/K^{\langle g\rangle}}(v_2)}{\N_{K/K^{\langle g\rangle}}(v_1)}=
          \frac{\prod_{i=0}^{p-1} g^i(v_2)}{\prod_{i=0}^{p-1} g^i(v_1)}\\
          \nonumber    &=&\frac{\prod_{i=0}^{p-1} (\sigma_1 \sigma_2)^i(v_2)}{\prod_{i=0}^{p-1} (\sigma_1 \sigma_2)^i(v_1)}=
          \frac{\prod_{i=0}^{p-1} (\sigma_2)^i(v_2)}{\prod_{i=0}^{p-1} (\sigma_1)^i(v_1)}\\
          \nonumber    &=&\frac{\N_{K/F_1}(v_2)}{\N_{K/F_2}(v_1)}=1.
      \end{eqnarray}
    (2) \quad This follows from (1) and Hilbert $90$.
\end{proof}

We also need some computational results which we now prove.
 \begin{prop}\label{Prop}
 Define the following elements of $A_{v_1,v_2}$\rm: $$Z=z_1z_2, W=tz_2, X=x_1^{-1}x_2, Y=x_1.$$ Then we have:
\begin{enumerate}
  \item $X^p=\frac{a_2}{a_1}.$
  \item  $W^p \in F.$
  \item $W X=\rho XW.$
  \item $Y^p=a_1.$
  \item $Z^p \in F.$
  \item $Z Y=\rho YZ.$
  \item $Z X=XZ.$
  \item $Z W= WZ.$
  \item $YX=XY.$
  \item $YW=WY.$

\end{enumerate}

 \end{prop}

\begin{proof} \ \\
                  (1), (4) \quad Recall that $K=F[x_1,x_2]$ and $x_1^p=a_1$, $x_2^p=a_2$. Now we compute:
                            $X^p=\Big(\frac{x_2}{x_1}\Big)^p=\frac{a_2}{a_1}.$
                            Statement (4) follows by a similar computation.\\
                  (7), (8), (9), (10) \quad  We prove statement (8) and note that statements (7), (9) and (10) follow by similar computations.
                  We compute:
                  \begin{eqnarray}
                  \nonumber ZWZ^{-1}W^{-1}&=& z_1z_2tz_2(z_1z_2)^{-1}(tz_2)^{-1}\\
                  \nonumber &=& z_1z_2tz_2z_2^{-1}z_1^{-1}z_2^{-1}t^{-1}\\
                  \nonumber &=& z_1z_2tz_1^{-1}z_2^{-1}t^{-1}\\
                  \nonumber &=& z_1z_2z_1^{-1}z_2^{-1}\sigma_2 \sigma_1(t)t^{-1}=u^{-1}u=1.
                  \end{eqnarray}
                  (2), (5) \quad First notice that $W$ acts by conjugation on $K$ as $\sigma_2$, indeed for $k\in K$ we have
                            $$ WkW^{-1}=tz_2k(tz_2)^{-1}=tz_2kz_2^{-1}t^{-1}=\s_2(k).$$
                            As $\s_2$ is of order $p$, we see that $W^p$ acts by conjugation on $K$ as the identity, so $W^p$ commutes with $K$.
                            But $K$ is a maximal subfield, thus $\C_{A_{v_1,v_2}}(K)=K$, implying that $W^p\in K$.
                            Now $\Gal(K/F)=\langle g,\sigma_2\rangle =\langle \sigma_1 \sigma_2,\sigma_2\rangle$, and it is enough to show $W^p$ is invariant
                            under its action. Notice that $Z$ acts on $K$ by conjugation as $g=\s_1\s_2$. Hence,
                            $$g(W^p)=ZW^pZ^{-1}\stackrel{(8)}{=}W^p$$ and $$\sigma_2(W^p)=WW^pW^{-1}=W^p$$ and statement (2) follows. \\
                            Statement (5) follows by a similar computation.\\
                  (3), (6) \quad We compute:
                  \begin{eqnarray}
                  \nonumber WX&=& tz_2x_1^{-1}x_2=x_1^{-1}tz_2x_2=x_1^{-1}t\rho x_2z_2\\
                  \nonumber &=& \rho x_1^{-1}x_2tz_2=\rho XW.
                  \end{eqnarray} Statement (6) follows by a similar computation.
                  \end{proof}
 \ \\
Finally we have:
\begin{thm}\label{MTB}
$$A_{v_1,v_2}\cong (a_1,W^{-p})_p\otimes (a_2,W^pZ^p)_p.$$
\end{thm}
\begin{proof}
Define two subalgebras of $A_{v_1,v_2}$: $$A_1=F[X,W] \ \ \ , \ \ \ A_2=F[Y,Z].$$
First note that by $(1)-(3)$ of Proposition \ref{Prop} we see that $A_1=(\frac{a_2}{a_1},W^p)_p$ and
by $(4)-(6)$ of the same Proposition we see that $A_2=(a_2,Z^p)_p.$
Now from $(7)-(10)$ of Proposition \ref{Prop} we get that $A_1,A_2$ commute element-wise.
Thus, by dimension count and the Double Centralizer Theorem (see \cite[Theorem 24.32]{Rowen}) we get that
$$A_{v_1,v_2}\cong A_1\otimes A_2=(a_1^{-1}a_2,W^p)_p\otimes(a_2,Z^p)_p$$ and
the Theorem follows.
\end{proof}

\section{Triple Massey products and abelian crossed products}
In this section we connect triple Massey products and abelian crossed products.\\
Recall from subsection 2.3 that given $\chi_a,\chi_b,\chi_c\in H^1(G_F)$ such that $$\chi_a\cup \chi_b=\chi_b\cup \chi_c=0,$$ a defining system is a subset
$\varphi=\{ \varphi_{a,b},\varphi_{a,b}\}\subset C^1(G_F)$ such that $$\partial(\varphi_{a,b})=\chi_a\cup \chi_b \hbox{ and } \partial(\varphi_{b,c})=\chi_b\cup \chi_c$$ in $C^2(G_F).$
For every defining system one constructs $$\langle \chi_a,\chi_b,\chi_c \rangle_{\varphi}=\chi_a\cup \varphi_{b,c}+\varphi_{a,b}\cup \chi_c$$ and the triple Massey product of $\chi_a,\chi_b,\chi_c$ is
$$\langle \chi_a,\chi_b,\chi_c \rangle=\{ \langle \chi_a,\chi_b,\chi_c \rangle_{\varphi}| \ \varphi \hbox{ \ is a defining system} \} \subseteq H^2(G_F).$$
\begin{prop}\label{split} For a defining system $\varphi$ as above we have:
\begin{enumerate}

\item $\res_{\K(\chi_a)}(\langle \chi_a,\chi_b,\chi_c \rangle_{\varphi})=\varphi_{a,b}\cup \chi_c\in H^2(\K(\chi_a)).$
\item $\res_{\K(\chi_c)}(\langle \chi_a,\chi_b,\chi_c \rangle_{\varphi})=\chi_a\cup \varphi_{b,c}\in H^2(\K(\chi_c)).$
\item $\res_{\K(\chi_a)\cap \K(\chi_c)}(\langle \chi_a,\chi_b,\chi_c \rangle_{\varphi})=0.$

\end{enumerate}
\end{prop}

\begin{proof}
(1) \quad Let $g_1,g_2 \in \K(\chi_a)$ one has:\begin{eqnarray}
  \nonumber \langle \chi_a,\chi_b,\chi_c \rangle_{\varphi}(g_1,g_2)&=&(\chi_a\cup\varphi_{b,c}+ \varphi_{a,b}\cup \chi_c)(g_1,g_2) \\
  \nonumber  &=& \chi_a\cup\varphi_{b,c}(g_1,g_2)+\varphi_{a,b}\cup \chi_c(g_1,g_2) \\
  \nonumber  &=& \chi_a(g_1)\cdot \varphi_{b,c}(g_2)+ \varphi_{a,b}(g_1)\cdot \chi_c(g_2)\\
  \nonumber  &=& 0\cdot\varphi_{b,c}(g_2)+\varphi_{a,b}(g_1)\cdot \chi_c(g_2) =\varphi_{a,b}\cup \chi_c(g_1,g_2).
\end{eqnarray}
Thus, $\res_{\K(\chi_a)}(\langle \chi_a,\chi_b,\chi_c \rangle_{\varphi})=\varphi_{a,b}\cup \chi_c\in H^2(\K(\chi_a)).$\\
Statement (2) follows from a similar computation.

(3) \quad Let $g_1,g_2 \in \K(\chi_a)\cap \K(\chi_c)$ one has:
\begin{eqnarray}
  \nonumber \langle \chi_a,\chi_b,\chi_c \rangle_{\varphi}(g_1,g_2)&=&(\chi_a\cup\varphi_{b,c}+ \varphi_{a,b}\cup \chi_c)(g_1,g_2) \\
  \nonumber  &=& \chi_a\cup\varphi_{b,c}(g_1,g_2)+\varphi_{a,b}\cup \chi_c(g_1,g_2) \\
  \nonumber  &=& \chi_a(g_1)\cdot \varphi_{b,c}(g_2)+ \varphi_{a,b}(g_1)\cdot \chi_c(g_2)\\
  \nonumber  &=& 0\cdot\varphi_{b,c}(g_2)+\varphi_{a,b}(g_1)\cdot 0 =0.
\end{eqnarray}
Thus, $\res_{\K(\chi_a)\cap \K(\chi_c)}(\langle \chi_a,\chi_b,\chi_c \rangle_{\varphi})=0$ and we are done.\end{proof}

\begin{prop}\label{ACPR}
Every $\langle \chi_a,\chi_b,\chi_c \rangle_{\varphi}\in \langle \chi_a,\chi_b,\chi_c \rangle$ corresponds to a class $[A]\in \Br(F)$, which is represented by a $\mathbb{Z}/p\mathbb{Z} \times \mathbb{Z}/p\mathbb{Z}$ abelian crossed product, $A$.
\end{prop}

\begin{proof}
Let $[A]\in \Br(F)$ correspond to $\langle \chi_a,\chi_b,\chi_c \rangle_{\varphi}$ under the isomorphism~(\ref{BrH2}).
By Remark \ref{split} we have $\res_{\K(\chi_a)\cap \K(\chi_c)}(\langle \chi_a,\chi_b,\chi_c \rangle_{\varphi})=0$. Thus $[A]$ is split by the field extension corresponding to $\K(\chi_a)\cap \K(\chi_c)$, which is $F[a^{\frac{1}{p}}, c^{\frac{1}{p}}]$ with Galois group isomorphic to $\mathbb{Z}/p\mathbb{Z} \times \mathbb{Z}/p\mathbb{Z}$. Hence by \cite[Corollary 24.37]{Rowen}, $[A]$ has a representative, a $F$-central simple algebra $A$, which is a $\mathbb{Z}/p\mathbb{Z} \times \mathbb{Z}/p\mathbb{Z}$ abelian crossed product.
\end{proof}

\begin{rem}
In \cite{EM}, the authors show that Proposition \ref{ACPR}, together with an old Theorem Albert (see \cite{Alb}), strengthened by Rowen (\cite[Corollary 5]{Row}), has an immediate corollary
stating that when $p=2$, $G_F$ has the Vanishing triple Massey product property.

\end{rem}

\section{Computation of $\varphi_{a,b}$}

Knowing Proposition \ref{ACPR}, was enough for the $p=2$ case. However, when $p\geq 3$ there is no generalization of Albert's Theorem.
In fact, works of Tignol and McKinnie (see \cite{Tig}, \cite{Mck} respectively), show that for $p\geq 3$ there exist indecomposable abelian crossed products of
exponent $p$ and degree $p^2$, thus can not be similar to the tensor product of two degree $p$ symbol algebras.
Hence, in order to solve the general case we need to get more information on the triple Massey product.
In particular, we will need to better understand $\varphi_{a,b},\varphi_{b,c}$ in the definition of the triple Massey product. To this end we assume
$$\chi_a\cup \chi_b=0,$$ and find explicit $\varphi_{a,b}\in C^1(G_F)$ such that $$\partial( \varphi_{a,b})=\chi_a\cup \chi_b.$$

Let $a,b\in F^{\times}$ be such that their corresponding Kummer characters, $\chi_a,\chi_b$ are linearly independent in $H^2(G_F).$
Assume that $\chi_a\cup \chi_b=0$ in $H^2(G_F)$. Let $$F_a=F[x_a|x_a^p=a] \ \ \ , \ \ \ F_b=F[x_b|x_b^p=b]$$ be the Kummer extensions with Galois groups
$$G_a=\langle\sigma_a \rangle\cong \mathbb{Z}/p\mathbb{Z} \hbox{ \  , \ } G_b=\langle\sigma_b\rangle\cong \mathbb{Z}/p\mathbb{Z};$$
where we have $\sigma_a(x_a)=\rho x_a$ and $\sigma_b(x_b)=\rho x_b$.\\
Let $L=F_aF_b$ with Galois group $$G_{ab}=\langle \sigma_a, \sigma_b \rangle \cong \mathbb{Z}/p\mathbb{Z}\times \mathbb{Z}/p\mathbb{Z}.$$
By Wedderburn \cite[Remark 24.46]{Rowen} we know there exists $v\in F_b$ such that $\N_{F_b/F}(v)=a$.
For $i=0,1,...,p-1$ let $v_i=\sigma_b^i(v)$ and define $$w=\prod_{i=0}^{p-1} v_i^i \in F_b.$$
Note that by Proposition \ref{prop3} we have:$$\sigma_b(w)=\frac{v_0^p}{a}w \hbox{ \  and  \ } \sigma_a(w)=w.$$
We want to construct a Galois extension containing $L$ and $w^{\frac{1}{p}}$.
\begin{prop}
$w$ is not a $p$-th power in $L$.
\end{prop}
\begin{proof}
Assume on the contrary that $w=t^p$ for some $t\in L$.
We compute:
$$\sigma_b(t)^p = \sigma_b(w)=\frac{v_0^p}{a}w =\Big(\frac{v_0}{x_a}\Big)^pw=\Big(\frac{v_0}{x_a}\Big)^pt^p.$$
Thus we get $$\Big( \frac{\sigma_b(t)}{t}\Big) ^p= \Big( \frac{v_0}{x_a}\Big) ^p. $$
Now by our assumption $L$ is a field, hence $$\frac{\sigma_b(t)}{t}=\rho^i \frac{v_0}{x_a}$$ for some $i$.
In particular \begin{equation}\label{1} \frac{\sigma_b(t)}{t}\in x_a^{-1}F_b. \end{equation}
On the other hand we have:
$$\sigma_a(t)^p = \sigma_a(w)=w=t^p.$$
Thus we get: $\Big(\frac{\sigma_a(t)}{t} \Big)^p=1$.
Again by our assumption $L$ is a field, hence $$\frac{\sigma_a(t)}{t}=\rho^j$$ for some $j$.
Hence, $t=x_a^ju$ for $u\in L^{\sigma_a}=F_b$ as an eigenvector of $\sigma_a$ with eigenvalue $\rho^j$. Now we compute:
$$\frac{\sigma_b(t)}{t}=\frac{\sigma_b(x_a^ju)}{x_a^ju}=\frac{\sigma_b(u)}{u}.$$
Thus \begin{equation}\label{2} \frac{\sigma_b(t)}{t}\in F_b. \end{equation}
But as $x_a\notin F_b$ we get a contradiction with (1) and we are done.\end{proof}

\begin{cor}
$K=L[x_w|x_w^p=w]$ is a field extension of $F$ of dimension $p^3$.
\end{cor}

\begin{thm}
$K/F$ is Galois.
\end{thm}

\begin{proof}
To prove the theorem we need to produce $p^3$ distinct $F$-automorphisms of $K$.
Clearly we have the automorphism $\tau$ defined by $$\tau(x_w)=\rho x_w, \ \ \ \tau(x_a)=x_a, \ \ \ \tau(x_b)=x_b.$$
Also as $w\in F_b$ we have $K\cong F_b[x_w]\otimes F_a$. Thus we can extend $\sigma_a$ to $K$ by setting $\sigma_a(x_w)=x_w$.

The hard part is to extend $\sigma_b$ from $L$ to $K$.
However as we know $$\sigma_b(w)=\frac{v_0^p}{a}w=\Big(\frac{v_0}{x_a}\Big)^pw,$$ we get that the minimal polynomial of $\sigma_b(x_w)$ over $L$ in a Galois closure is the conjugate by $\sigma_b$ of the minimal polynomial of $x_w$, namely, $$\sigma_b(X^p-w)=X^p-\Big(\frac{v_0}{x_a}\Big)^pw.$$
Thus, all we have to do is send $x_w$ to a root of $X^p-\Big(\frac{v_0}{x_a}\Big)^pw$.
Now as the roots of $X^p-\Big(\frac{v_0}{x_a}\Big)^pw$ are: $$\{\rho^i \frac{v_0}{x_a}x_w|i=0,1,...,p-1\};$$ all we need to do is choose one of them.
We choose to extend $\sigma_b$ by setting $$\sigma_b(x_w)=\frac{v_0}{x_a}x_w.$$
Thus we have produced all the necessary automorphisms and we are done.
\end{proof}

\begin{prop}
$\Gal(K/F)=\langle\sigma_a,\sigma_b,\tau \rangle$ which we denote by $G$, has the following relations:
\begin{enumerate}
\item $\sigma_b\sigma_a=\tau \sigma_a\sigma_b$.
\item The center, $\Cent(G)$, of $G$ is $\langle \tau\rangle$.
\end{enumerate}
\end{prop}

\begin{proof}
(1)   \quad
        It is enough to check (1) on the generators of $K$ over $F$.
        We have: \begin{enumerate}
                     \item[(i)] $\sigma_b\sigma_a(x_b)=\sigma_b(x_b)=\rho x_b=\tau(\rho x_b)=\tau\sigma_a(\rho x_b)=\tau\sigma_a\sigma_b(x_b).$
                     \item[(ii)] $\sigma_b\sigma_a(x_a)=\sigma_b(\rho x_a)=\rho x_a=\tau(\rho x_a)=\tau\sigma_a(x_a)=\tau\sigma_a\sigma_b(x_a).$
                     \item[(iii)] $\sigma_b\sigma_a(x_w)=\sigma_b(x_w)=\frac{v_0}{x_a}x_w=\tau(\rho^{-1}\frac{v_0}{x_a}x_w)=\tau\sigma_a(\frac{v_0}{x_a}x_w)=
                    \tau\sigma_a\sigma_b(x_w).$
                 \end{enumerate}
        Thus we have proved that $\sigma_b\sigma_a=\tau \sigma_a\sigma_b$.\\
(2) \quad
        First we show $\langle \tau\rangle \subseteq \Cent(G)$:\\
        Again, we check on the generators of $K$ over $F$.\\
        We first show: $\sigma_a\tau=\tau\sigma_a.$
        \begin{enumerate}
            \item[(i)] $\sigma_a\tau(x_a)=\sigma_a(x_a)=\rho x_a=\tau(\rho x_a)=\tau\sigma_a(x_a).$ \\ The same computation shows that $\sigma_a\tau(x_b)=\tau\sigma_a(x_b).$
            \item[(ii)] $\sigma_a\tau(x_w)=\sigma_a(\rho x_w)=\rho x_w=\tau(x_w)=\tau\sigma_a(x_w).$
        \end{enumerate}
    Thus we have proved that $\sigma_a\tau=\tau \sigma_a.$ \\
    As for $\sigma_b\tau=\tau\sigma_b$ we compute:
    \begin{enumerate}
        \item[(i)] The exact same computation as above shows that \\ $\sigma_b\tau(x_a)=\tau\sigma_b(x_a)$ and $\sigma_b\tau(x_b)=\tau\sigma_b(x_b).$
        \item[(ii)] $\sigma_b\tau(x_w)=\sigma_b(\rho x_w)=\rho \frac{v_0}{x_a}x_w=\tau(\frac{v_0}{x_a}x_w)=\tau\sigma_b(x_w).$
    \end{enumerate}
    Thus we proved that $\sigma_b\tau=\tau\sigma_b$, that is $\langle\tau\rangle \subseteq \Cent(G)$.
    We now move to proving $\Cent(G)\subseteq \langle \tau\rangle.$
    To this end let $g=\sigma_b^i\sigma_a^j\tau^k\in \Cent(G)$, then we have, $\sigma_bg=g\sigma_b$. But on the one hand $$\sigma_bg=\sigma_b^{i+1}\sigma_a^j\tau^k.$$
    On the other hand we have: $$ g\sigma_b= \sigma_b^i\sigma_a^j\tau^k\sigma_b=\sigma_b^{i+1}\sigma_a^j\tau^{k-j}.$$ Thus we get $j=0$. A similar computation shows that $i=0$.
    Hence we showed $\Cent(G)\subseteq \langle \tau\rangle$ and part (2) follows.\end{proof}

\begin{rem}\label{proj}
There are projections:
$$\pi_a:G\rightarrow G_a \ \ \ \hbox{such that} \ \ \ \pi_a(\sigma_a)=\sigma_a \ \ \ \hbox{with kernel} \ \ \ \langle \sigma_b, \tau \rangle.$$
$$\pi_b:G\rightarrow G_b \ \ \ \hbox{such that} \ \ \ \pi_b(\sigma_b)=\sigma_b \ \ \ \hbox{with kernel} \ \ \ \langle \sigma_a, \tau \rangle.$$
\end{rem}

We are now ready to define a map $$\varphi_{a,b}\in C^1(G) \ \ \ \hbox{such that} \ \ \ \partial(\varphi_{a,b})=\chi_a\cup \chi_b \in C^2(G).$$
Let $\chi_a,\chi_b \in H^1(G)$ be the inflations to $G$, with respect to $\pi_a,\pi_b$, of the
Kummer characters $\chi_a\in H^1(G_a),\chi_b\in H^1(G_b)$ corresponding to $a,b\in F^{\times}$.
\begin{rem}
Note that by the above we have: $$\K(\chi_a)=\langle \sigma_b, \tau \rangle \ \ \ ; \ \ \ \K(\chi_b)=\langle \sigma_a, \tau \rangle .$$
\end{rem}

\begin{lem}
Consider $\chi_a,\chi_b$ as elements of $C^1(G).$ \\
Let $g_1=\sigma_b^i\sigma_a^j\tau^k, \ g_2=\sigma_b^r\sigma_a^s\tau^t \in G.$ We have:
$$\chi_a\cup\chi_b(g_1,g_2)=jr  \ \ \ \hbox{in} \ C^2(G).$$
\end{lem}

\begin{proof}
We compute: \\ $\chi_a\cup\chi_b(g_1,g_2)= \chi_a(g_1)\chi_b(g_2)= \chi_a(\pi_a(g_1))\chi_b(\pi_b(g_2))=$\\$\chi_a(\sigma_a^j)\chi_b(\sigma_b^r)=jr.$
\end{proof}
Define $\varphi_{a,b}\in C^1(G)$ via $$\varphi_{a,b}(\sigma_b^i\sigma_a^j\tau^k)=k.$$
\begin{prop}\label{cup}
$\partial (\varphi_{a,b})=\chi_a\cup\chi_b$ in $C^2(G)$, that is $\chi_a\cup\chi_b=0$ in $H^2(G).$
\end{prop}

\begin{proof}
Let $g_1,g_2\in G$ be as above. We compute:\begin{eqnarray}
\nonumber \partial(\varphi_{a,b})(g_1,g_2)&=& \partial(\varphi_{a,b})(\sigma_b^i\sigma_a^j\tau^k,\sigma_b^r\sigma_a^s\tau^t)\\
\nonumber &=& \varphi_{a,b}(\sigma_b^i\sigma_a^j\tau^k)+\varphi_{a,b}(\sigma_b^r\sigma_a^s\tau^t)-
                \varphi_{a,b}(\sigma_b^i\sigma_a^j\tau^k\sigma_b^r\sigma_a^s\tau^t) \\
\nonumber &=& k+t-\varphi_{a,b}(\sigma_b^{i+r}\sigma_a^{j+s}\tau^{k+t-rj})=rj\\
\nonumber &=& \chi_a\cup\chi_b(g_1,g_2).
\end{eqnarray}
Thus $\partial (\varphi_{a,b})=\chi_a\cup\chi_b$ in $C^2(G).$
\end{proof}

\begin{cor}
Let $\chi_w\in H^1(\K(\chi_b))$ be the Kummer character corresponding to $w\in F_b$, that is
$$\chi_w(\sigma_a)=0 \ \ \ ; \ \ \ \chi_w(\tau)=1$$ and consider it also as an element of $C^1(\K(\chi_b)).$
We have that $$\res_{\K(\chi_b)}(\varphi_{a,b})=\chi_w \in H^1(\K(\chi_b)).$$
\end{cor}
\begin{proof}
Let $g=\sigma_a^i\tau^j \in \K(\chi_b)$, we compute:
\begin{eqnarray}
\nonumber \res_{\K(\chi_b)}(\varphi_{a,b})(\sigma_a^i\tau^j)&=&\varphi_{a,b}(\sigma_b^0\sigma_a^i\tau^j)\\
\nonumber &=&j=\chi_w(\sigma_a^i\tau^j).
\end{eqnarray}
Thus $\res_{\K(\chi_b)}(\varphi_{a,b})=\chi_w \in H^1(\K(\chi_b)).$
\end{proof}

\section{Proof of the main Theorem }
In this section we will prove Theorem \ref{Main} stating that for any prime $p$, $G_F$  has the Vanishing triple product property, using the previous sections.

Recall that we are given $ \chi_a,\chi_b,\chi_c\in H^1(G_F) $ such that  $$\chi_a\cup\chi_b=\chi_b\cup\chi_c=0.$$
Then, there exist $\varphi_{a,b},\varphi_{b,c}\in C^1(G_F)$ such that
$$\partial(\varphi_{a,b})=\chi_a\cup \chi_b \ \ \ \hbox{and} \ \ \ \partial(\varphi_{b,c})=\chi_b\cup \chi_c$$ in $\C^2(G_F)$. Hence,
$$\langle \chi_a,\chi_b,\chi_c \rangle_{\varphi}= \chi_a\cup\varphi_{b,c}+ \varphi_{a,b}\cup \chi_c$$ is in $H^2(G_F)$, and
$$\langle \chi_a,\chi_b,\chi_c \rangle=\langle \chi_a,\chi_b,\chi_c \rangle_{\varphi}+\chi_a\cup H^1(G_F)+H^1(G_F)\cup \chi_c.$$

In order to use section $3$, we need to assume $\chi_a,\chi_c$ are linearly independent, and in order to use section $5$ for $\chi_b\cup \chi_c=0$ to find $\varphi_{b,c}$ we want to assume the last two slots ($\chi_b,\chi_c$) are linearly independent.

\begin{prop}
It is enough to prove Theorem \ref{Main} for the case: $\chi_a,\chi_c$ are linearly independent, and $\chi_b,\chi_c$ are linearly independent.
\end{prop}
\begin{proof}
Assume we know Theorem \ref{Main} for this case and consider the other cases.\\
{\bf{Case $1.$}}  $\chi_a,\chi_c$ are linearly dependent.\\ Without loss of generality we may assume $\chi_a=t\chi_c$ for some $t\in \mathbb{F}_p$.
It follows that $\K(\chi_c)\subseteq\K(\chi_a)$. Now given a defining system $\varphi=\{\varphi_{a,b},\varphi_{b,c}\}$, we compute,\\
$\res_{\K(\chi_c)}(\langle \chi_a,\chi_b,\chi_c \rangle_{\varphi})=\res_{\K(\chi_c)}(\chi_a)\cup \res_{\K(\chi_c)}(\varphi_{b,c})+ \res_{\K(\chi_c)}(\varphi_{a,b})\cup \res_{\K(\chi_c)}(\chi_c)=0.$\\
By \cite[Corollary 24.37]{Rowen}, we get that $\langle \chi_a,\chi_b,\chi_c \rangle_{\varphi}\in \chi_c\cup H^1(G_F)$, and $0\in \langle \chi_a,\chi_b,\chi_c \rangle.$\\
{\bf{Case $2.$}} $\chi_b,\chi_c$ are linearly dependent. \\ First note that by the first case we may assume $\chi_a,\chi_c$ are linearly independent.
This implies that either $\chi_b=0$ or $\chi_a,\chi_b$ are linearly independent.
If $\chi_b=0$ we see that $\chi_a\cup \chi_b=0; \chi_b\cup \chi_c=0$ in $C^2(G_F)$, so we can choose $\varphi_{a,b}=\varphi_{b,c}=0$ in $C^1(G_F)$ as a defining system and we get $0\in \langle \chi_a,\chi_b,\chi_c \rangle$.\\
The remaining case is $\chi_a,\chi_b$ are linearly independent. Note that by assumption $\langle \chi_a,\chi_b,\chi_c \rangle$ is defined, which implies that $\langle \chi_c,\chi_b,\chi_a\rangle$ is defined. Moreover, by the assumption of the Theorem we have $0\in \langle \chi_c,\chi_b,\chi_a\rangle.$
Let $\varphi=\{\varphi_{c,b}, \varphi_{b,a}\}$ be a defining system for $\langle \chi_c,\chi_b,\chi_a\rangle$ such that $\langle \chi_c,\chi_b,\chi_a\rangle_{\varphi}=0.$ Consider the Galois group $G_{ab}$, of $F[a^{\frac{1}{p}},b^{\frac{1}{p}}]$ which by assumption is $\langle \s_a, \s_b\rangle \cong \mathbb{Z}\times\mathbb{Z}$ with the obvious action. Define two functions, $\psi_1, \psi_2\in C^1(G_{ab})$ by,
$$\psi_1(\s_a^i\s_b^j)=-ij, \ \psi_2(\s_a^i\s_b^j)=-j^2.$$ Note that we have $\res_{\K(\chi_a)}(\psi_1)=0$. A direct computation shows that,
$$\partial(\psi_1)=\chi_a\cup\chi_b+\chi_b\cup\chi_a, \ \  \partial(\psi_2)=2\chi_b\cup\chi_b.$$ Now let $t\in \mathbb{F}_p$ be such that $\chi_c=t\chi_b$ and define $\varphi'=\{\varphi_{a,b}, \varphi_{b,c}\}$ by $$\varphi_{a,b}=\psi_1-\varphi_{b,a}, \ \ \varphi_{b,c}=t\psi_2-\varphi_{c,b}.$$
We compute:\begin{eqnarray}
             \nonumber \partial(\varphi_{a,b}) &=& \partial(\psi_1)-\partial(\varphi_{b,a}) \\
             \nonumber &=& \chi_a\cup\chi_b+\chi_b\cup\chi_a - \chi_b\cup\chi_a=\chi_a\cup\chi_b.
           \end{eqnarray}
\begin{eqnarray}
             \nonumber \partial(\varphi_{b,c}) &=& \partial(t\psi_2)-\partial(\varphi_{c,b}) \\
             \nonumber &=& 2t\chi_b\cup\chi_b-t\chi_b\cup\chi_b=t\chi_b\cup\chi_b=\chi_b\cup\chi_c.
           \end{eqnarray}
Thus we see $\varphi'$ is a defining system for $\langle \chi_a,\chi_b,\chi_c \rangle$.
We compute, \begin{eqnarray}
              \nonumber \langle \chi_a,\chi_b,\chi_c \rangle_{\varphi'} &=& \chi_a\cup(t\psi_2-\varphi_{c,b})+(\psi_1-\varphi_{b,a})\cup\chi_c \\
              \nonumber &=& t\chi_a\cup \psi_2+\psi_1\cup\chi_c-(\chi_a\cup\varphi_{c,b}+\varphi_{b,a}\cup\chi_c)
            \end{eqnarray}
%

 Now, \begin{eqnarray}
        \nonumber \res_{\K(\chi_a)}(\langle \chi_a,\chi_b,\chi_c \rangle_{\varphi^{'}})&=& t\cdot\res_{\K(\chi_a)}(\chi_a)\cup \res_{\K(\chi_a)}(\psi_2) + \\
        \nonumber & & \res_{\K(\chi_a)}(\psi_1)\cup \res_{\K(\chi_a)}(\chi_c) - \\
        \nonumber & & \res_{\K(\chi_a)}(\chi_a)\cup \res_{\K(\chi_a)}(\varphi_{c,b}) - \\
        \nonumber & & \res_{\K(\chi_a)}(\varphi_{b,a})\cup\res_{\K(\chi_a)}(\chi_c) \\
        \nonumber &=& \res_{\K(\chi_a)}(\varphi_{c,b})\cup\res_{\K(\chi_a)}(\chi_a) + \\
        \nonumber & & \res_{\K(\chi_a)}(\chi_c)\cup\res_{\K(\chi_a)}(\varphi_{b,a}) \\
        \nonumber &=& \res_{\K(\chi_a)}(\langle \chi_c,\chi_b,\chi_a\rangle_{\varphi})=0.
      \end{eqnarray}
      Thus by \cite[Corollary 24.37]{Rowen}, we get $$\langle \chi_a,\chi_b,\chi_c \rangle_{\varphi^{'}}\in \chi_a\cup H^1(G_F),$$ and $0\in \langle \chi_a,\chi_b,\chi_c \rangle.$
      \end{proof}

From now we assume $\chi_a,\chi_b,\chi_c$ satisfy $\chi_a,\chi_c$ are linearly independent, and $\chi_b,\chi_c$ are linearly independent.
Applying the previous subsection, we see there is a defining system $\varphi=\{\varphi_{a,b},\varphi_{b,c}\}$ such that
$$\Res_{\K(\chi_{c})}(\varphi_{{b},{c}})=\chi_{w_c}$$ where $$w_c = \prod _{i=0}^{p-1} \sigma_c^i(v)^i\in F_{c}$$
for $v \in F_c$ such that $\N_{F_c/F}(v)=b.$\\

Let $A$ be the corresponding degree $p^2$ and exponent $p$ abelian crossed product as in \ref{ACPR}, we have:
\begin{prop}\label{res}
$[\Res_{F_c}(A)]= [(a,w_c)_{p,F_c}]$
\end{prop}

\begin{proof}
By the naturality of the isomorphism $\Psi$ from the introduction we have that, $[\Res_{F_c}(A)]$ corresponds to
$\Res_{\K_{\chi_c}}(\langle \chi_a,\chi_b,\chi_c \rangle_{\varphi})=\Res_{\K_{\chi_c}}(\chi_a)\cup \Res_{\K_{\chi_c}}(\varphi_{b,c})=
\Res_{\K_{\chi_c}}(\chi_a)\cup \chi_{w_c},$ which in turn corresponds to $[(a,w_c)_{p,F_c}]$.\end{proof}

Note that since $\chi_a\cup \chi_b=0$, we know there exist $u\in F_a$ such that $\N_{F_a/F}(u)=b$, and we let $$w_a=\prod _{i=0}^{p-1} \sigma_a^i(u)^i\in F_{a}.$$
Thus, we can construct the abelian crossed product from section~$2$,
 $$A_{u,v}=\Big(F_aF_c,\{\sigma_a,\sigma_c\}, \{w_c,w_a\}, \frac{v}{u}\Big).$$
Recall this means $A_{u,v}=F_aF_c[z_a,z_c]$ with the following relations:
$$z_akz_a^{-1}=\sigma_a(k), \  z_ckz_c^{-1}=\sigma_c(k),  \ z_a^p=w_c, \ z_c^p=w_a, \ z_cz_a=\frac{v}{u}z_az_c.$$
The next step is to show $[A]$ and $[A_{u,v}]$ only differ by the class of a symbol algebra $(c,s)_{F,p}$ for some $s\in \mul{F}.$
To this end we first notice that:
\begin{prop}
$[\Res_{F_c}(A_{u,v})]=[(a,w_c)_{p,F_c}].$
\end{prop}

\begin{proof}

By \cite[Corollary 24.24]{Rowen} we have $[\Res_{F_c}(A_{u,v})]=[\C_{A_{u,v}}(F_c)]=$\\ $[F_aF_c[z_a]]=[(a,w_c)_{p,F_c}]$ as needed.
\end{proof}

\begin{cor}\label{x}
We have: $[A_{u,v}]=[A]\otimes [(c, s)_{p,F}]$ for some $s\in \mul{F}.$
\end{cor}

\begin{proof}
By \cite[Corollary 24.37]{Rowen} it is enough to show that $[A_{u,v}]\otimes [A]^{-1}$ is split by $F_c.$
But from the above we have $\Res_{F_c}([A_{u,v}]\otimes [A]^{-1})=[(a,w_c)_{p,F_c}]\otimes [(a,w_c)_{p,F_c}]^{-1}=1.$
\end{proof}

Corollary \ref{x} tells us that:

\begin{prop}\label{B}
The abelian crossed product $A_{u,v}$ represents an element of $\langle \chi_a,\chi_b,\chi_c \rangle$.
\end{prop}

\begin{proof}
This is clear in light of Corollary \ref{x} and Remark \ref{Set}.
\end{proof}

We can finally prove:
\begin{thm}\label{Main}
For any prime $p$ and any field $F$ containing a primitive $p$-th root of unity, $G_F$ has the Vanishing triple Massey product property.
\end{thm}

\begin{proof}
By Remark \ref{Set}, it is enough to show that there exist an element $\langle \chi_a,\chi_b,\chi_c \rangle_{\varphi}\in \langle \chi_a,\chi_b,\chi_c \rangle$
such that $\langle \chi_a,\chi_b,\chi_c \rangle_{\varphi}\in \chi_a\cup H^1(G_F) + H^1(G_F)\cup \chi_c$.
Now by Proposition \ref{B}, we know the abelian crossed product $A_{u,v}$ represents some element
$\langle \chi_a,\chi_b,\chi_c \rangle_{\varphi}\in \langle \chi_a,\chi_b,\chi_c \rangle$.
Thus it is enough to show that $A_{u,v}\cong (a,s)_{F,p}\otimes (c,t)_{F,p}$ for some $s,t\in \mul{F}$, but this is exactly Theorem \ref{MTB}.
\end{proof}

\end{document}